\documentclass[11pt,reqno]{amsart}

\usepackage{amsmath,amsthm,amssymb}

\usepackage{graphics}
\usepackage{hyperref}
\usepackage[usenames, dvipsnames]{xcolor}

\usepackage[square,sort,comma,numbers]{natbib}

\definecolor{darkblue}{rgb}{0.0,0.0,0.3}
\hypersetup{colorlinks,breaklinks,
  linkcolor=darkblue,urlcolor=darkblue,
anchorcolor=darkblue,citecolor=darkblue}

\theoremstyle{plain}
\newtheorem{theorem}{Theorem}[section]
\newtheorem*{theorem*}{Theorem}
\newtheorem{lemma}[theorem]{Lemma}
\newtheorem{proposition}[theorem]{Proposition}
\newtheorem*{proposition*}{Proposition}
\newtheorem{corollary}[theorem]{Corollary}
\newtheorem*{corollary*}{Corollary}

\theoremstyle{definition}
\newtheorem{remark}[theorem]{Remark}

\numberwithin{equation}{section}

\renewcommand{\Im}{\operatorname{Im}}
\renewcommand{\Re}{\operatorname{Re}}
\DeclareMathOperator{\SL}{SL}

\DeclareMathOperator*{\Res}{Res}

\title[Sums of Fourier Coefficients]{The Second Moment of Sums of Coefficients of Cusp Forms}
\author[Hulse]{Thomas A. Hulse}
\thanks{Research of the first author was supported by a Coleman Postdoctoral Fellowship at Queen's University.}

\author[Kuan]{Chan Ieong Kuan}

\author[Lowry-Duda]{David Lowry-Duda}
\thanks{The third author is supported by the National Science Foundation Graduate Research Fellowship Program under Grant No. DGE 0228243.}

\author[Walker]{Alexander Walker}

\begin{document}

\maketitle

\begin{abstract}
  Let $f$ and $g$ be weight $k$ holomorphic cusp forms and let $S_f(n)$ and $S_g(n)$ denote the sums of their first $n$ Fourier coefficients.
  Hafner and Ivi\'{c}~\cite{HafnerIvic89} proved asymptotics for $\sum_{n \leq X} \lvert S_f(n) \rvert^2$ and proved that the Classical Conjecture, that $S_f(X) \ll X^{\frac{k-1}{2} + \frac{1}{4} + \epsilon}$, holds on average over long intervals.

  In this paper, we introduce and obtain meromorphic continuations for the Dirichlet series $D(s, S_f \times S_g) = \sum S_f(n)\overline{S_g(n)} n^{-(s+k-1)}$ and $D(s, S_f \times \overline{S_g}) = \sum_n S_f(n)S_g(n) n^{-(s + k - 1)}$.
  We then prove asymptotics for the smoothed second moment sums $\sum  S_f(n)\overline{S_g(n)} e^{-n/X}$, giving a smoothed generalization of~\cite{HafnerIvic89}.
  We also attain asymptotics for analogous sums of normalized Fourier coefficients.
  Our methodology extends to a wide variety of weights and levels, and comparison with~\cite{chandrasekharan1962functional} indicates very general cancellation between the Rankin-Selberg $L$-function $L(s, f\times g)$ and convolution sums of the coefficients of $f$ and $g$.

  In forthcoming works, the authors apply the results of this paper to prove the Classical Conjecture on $\lvert S_f(n) \rvert^2$ is true on short intervals, and to prove sign change results on $\{S_f(n)\}_{n \in \mathbb{N}}$.
\end{abstract}

\section{Introduction and Statement of Results}

Let $f$ be a holomorphic cusp form on a congruence subgroup $\Gamma \subseteq \SL_2(\mathbb{Z})$ and of positive weight $k$, where $k \in \mathbb{Z}\cup(\mathbb{Z}+\frac{1}{2})$.
Let the Fourier expansion of $f$ at $\infty$ be given by
\[
  f(z) = \sum_{n \geq 1} a(n)e(nz),
\]
where $e(z) = e^{2\pi iz}$.
In this paper, we consider upper bounds for the second moment of the partial sums of the Fourier coefficients,
\[
  S_f(n) := \sum_{m \leq n}a(m).
\]

Bounds on the coefficients $a(n)$ are of great interest and have wide application.
The famous Ramanujan-Petersson conjecture, which was proven to hold for integral weight holomorphic cusp forms as a consequence of Deligne's proof of the Weil Conjecture~\cite{Deligne}, gives us that $a(n)\ll n^{\frac{k-1}{2} + \epsilon}$ and from this one might naively assume $S_f(X) \ll X^{\frac{k-1}{2} + 1 + \epsilon}$.
However, there is significant cancellation in the sum and we expect the far better bound,
\begin{equation}
  S_f(X) \ll X^{\frac{k-1}{2} + \frac{1}{4} + \epsilon},
\end{equation}
which we refer to as the ``Classical Conjecture,'' echoing Hafner and Ivi\'{c} in their work~\cite{HafnerIvic89}.

Chandrasekharan and Narasimhan, as a consequence of their much broader work on the average order of arithmetical functions~\cite{chandrasekharan1964mean},~\cite{chandrasekharan1962functional}, proved that the Classical Conjecture is true \emph{on average} by showing that
\begin{equation}\label{eq:OnAverageSquares}
  \sum_{n \leq X} \lvert S_f(n) \rvert^2 = CX^{k- 1 + \frac{3}{2}} + B(X),
\end{equation}
where $B(x)$ is an error term,
\begin{equation}\label{beebound}
  B(X) = \begin{cases}
    O(X^{k}\log^2(X)) \\
    \Omega\left(X^{k - \frac{1}{4}}\frac{(\log \log \log X)^3}{\log X}\right),
  \end{cases}
\end{equation}
and $C$ is the constant,
\begin{equation}\label{eq:constantCfromHaffner}
  C = \frac{1}{(4k + 2)\pi^2} \sum_{n \geq 1}\frac{\lvert a(n) \rvert^2}{n^{k + \frac{1}{2}}}.
\end{equation}
A application of the Cauchy-Schwarz inequality to~\eqref{eq:OnAverageSquares} leads to the on-average statement that
\begin{equation}
  \frac{1}{X} \sum_{n \leq X} |S_f(n)| \ll X^{\frac{k-1}{2} + \frac{1}{4}}.
\end{equation}
From this, Hafner and Ivi\'{c}~\cite{HafnerIvic89} were able to show that for holomorphic cusp forms with real coefficients of full integral weight on $\SL_2(\mathbb{Z})$,
\begin{equation}\label{eq:hafnerivic_onethird}
  S_f(X) \ll X^{\frac{k-1}{2} + \frac{1}{3}}.
\end{equation}
Their argument requires cusp forms of full level and the Ramanujan-Petersson Conjecture, but one can state analogous results for general level in terms of the corresponding best-known progress towards the conjecture.

Better lower bounds are known for $B(X)$.
In the same work,~\cite{HafnerIvic89}, Hafner and Ivi\'{c} improved the lower bound of Chandrasekharan and Narasimhan for full-integral weight forms of level one and showed that
\begin{equation}
  B(X) = \Omega\left(X^{k  - \frac{1}{4}}\exp\left(D \tfrac{(\log \log x )^{1/4}}{(\log \log \log x)^{3/4}}\right)\right),
\end{equation}
for a particular constant $D$.

In general, it is possible to translate upper bounds for $B(X)$ into upper bounds for $S_f(X)$.
By the Ramanujan-Petersson Conjecture, we have $\lvert S_f(X +\ell) \rvert^2 \gg \lvert S_f(X) \rvert^2$ for all $\ell \ll  \alpha(X)$ where $\alpha(X):=\lvert S_f(X)\rvert X^{-\frac{k-1}{2}-\epsilon}$, for any small $\epsilon > 0$, and $X$ sufficiently large.
It follows from~\eqref{eq:OnAverageSquares}, and the binomial expansion of $(X \pm \alpha(X))^{k+\frac{1}{2}}$ that
\begin{equation*}
  \frac{\lvert S_f(X) \rvert^3}{X^{\frac{k-1}{2}+\epsilon}} \ll \sum_{\lvert n - X \rvert \leq \alpha(X)} \lvert S_f(n) \rvert^2 = O\left( X^{\frac{k-1}{2} + \frac{1}{2} - \epsilon}\lvert S_f(X)\rvert + B(X)\right),
\end{equation*}
so that $\lvert S_f(X) \rvert \ll X^{\frac{k-1}{2} + \frac{1}{4}}$ or $\lvert S_f(X) \rvert \ll B(X)^{\frac{1}{3}}X^{\frac{k-1}{6}}$.
The first possibility is the Classical Conjecture and the second relates $B(X)$ to $S_f(X)$.
In particular, the bound $B(X) \ll X^{k}\log^2(X)$ from~\cite{chandrasekharan1962functional} nearly allows us to recover Hafner and Ivi\'{c}'s upper bound~\eqref{eq:hafnerivic_onethird}.

In this paper, we consider Dirichlet series associated to the norm-squared partial sums, $\lvert S_f(n) \rvert^2$, and  squared partial sums, $S_f(n)^2$,
\begin{equation}\label{eq:Dsff}
  D(s, S_f\times S_f) := \sum_{n \geq 1}\frac{\lvert S_f(n) \rvert^2}{n^{s + k - 1}}, \quad D(s, S_f\times \overline{S_f}) := \sum_{n \geq 1} \frac{S_f(n)^2}{n^{s + k - 1}}.
\end{equation}
The extra factor of $n^{k-1}$ appearing in the denominator serves to normalize the $L$-functions and shifted convolution sums that appear in the decomposition of these series, as we will see in Section~\ref{sec:decomposition}, so that every functional equation is of the form $s \mapsto 1-s$.
We choose our notation to be reminiscent of the notation for the Rankin-Selberg convolution.
We have been unable to find previous occurrences of the above series in the literature, but sufficient understanding of the analytic properties of $D(s, S_f \times \overline{S_f})$ and $D(s, S_f \times S_f)$ allows new insights and avenues for investigating the behaviour of $S_f(n)$.

More generally, we also study the Dirichlet series associated to the product of the partial sums of two weight $k$ holomorphic cusp forms,  $f$ and $g$,
\begin{equation}\label{eq:Dsfg}
  D(s, S_f \times S_g) := \sum_{n \geq 1}\frac{S_f(n) \overline{S_g(n)}}{n^{s+k-1}}, \quad D(s, S_f \times \overline{S_g}) := \sum_{n \geq 1}\frac{S_f(n)S_g(n)}{n^{s + k - 1}}.
\end{equation}
Initially we work only with $f,g$ of full-integral weight and of level one.
The methodology is almost identical for $D(s, S_f \times \overline{S_g})$ as it is for $D(s, S_f \times S_g)$, so we exposit only for $D(s, S_f \times S_g)$ and state the results for $D(s, S_f \times \overline{S_g})$.
In Section~\ref{sec:decomposition}, we decompose $D(s, S_f \times S_g)$ into pieces which we refer to as the \emph{diagonal} and \emph{off-diagonal} contributions.
Our main result is the meromorphic continuation of these \emph{diagonal} and \emph{off-diagonal} pieces, given by Theorem~\ref{thm:Wsfgmero} in Section~\ref{sec:Zswfg}, ultimately giving the meromorphic continuations of $D(s, S_f\times S_g)$ and $D(s, S_f \times \overline{S_g})$.

The major difficulty arises in determining the analytic behaviour of the \emph{off-diagonal}, which involves the shifted convolution sum
\[
  Z(s, w, f\times g) := \sum_{n,h \geq 1} \frac{a(n)\overline{b(n-h)}+a(n-h)\overline{b(n)}}{n^{s + k - 1}h^w},
\]
where the $a(n)$ and $b(n)$ are coefficients of $f$ and $g$, respectively. Our approach and notation to determine the analytic behaviour of $Z(s, w, f\times g)$ is similar to that in~\cite{HoffsteinHulse13}, though the technique originates with Selberg~\cite{Se}. In particular, the above symmetrized form of $Z(s,w,f \times g)$ is reminiscent of a construction devised by Bringmen, Mertens and Ono in \cite{Ono}.
In Section~\ref{sec:Zswfg}, we describe the behaviour of $Z(s, 0, f\times f)$ and show that there is remarkable cancellation at potential poles, both within the sum in its spectral expansion and with the pole due to the Rankin-Selberg $L$-function $L(s, f\times f)$.

We use an inverse Mellin transform and apply the meromorphic continuation of $D(s, S_f \times S_g)$ in Section~\ref{sec:second_moment_of_sums} to obtain the following smoothed partial sum result.

\begin{theorem}\label{thm:second_moment}  Suppose that $f$ and $g$ are holomorphic cusp forms of $\SL_2(\mathbb{Z})$ of integral weight $k$. Then for any $\epsilon > 0$,
  \begin{equation} \label{maineq1}
    \frac{1}{X} \sum_{n \geq 1}\frac{S_f(n)\overline{S_g(n)}}{n^{k - 1}}e^{-n/X} = CX^{\frac{1}{2}} + O_{f,g,\epsilon}(X^{-\frac{1}{2} + \theta + \epsilon})
  \end{equation}
  and
   \begin{equation} \label{maineq2}
    \frac{1}{X} \sum_{n \geq 1}\frac{S_f(n)S_g(n)}{n^{k - 1}}e^{-n/X} = C'X^{\frac{1}{2}} + O_{f,g,\epsilon}(X^{-\frac{1}{2} + \theta + \epsilon}),
  \end{equation}
  where
  \begin{equation*}
    C = \frac{\Gamma(\tfrac{3}{2}) }{4\pi^2} \frac{L(\frac{3}{2}, f\times g)}{\zeta(3)}= \frac{\Gamma(\tfrac{3}{2})}{4\pi ^2} \sum_{n \geq 1} \frac{a(n)\overline{b(n)}}{n^{k + \frac{1}{2}}},
  \end{equation*}
  and
    \begin{equation*}
    C' = \frac{\Gamma(\tfrac{3}{2})}{4\pi^2} \frac{L(\frac{3}{2}, f\times \overline{g})}{\zeta(3)}  = \frac{\Gamma(\tfrac{3}{2})}{4\pi ^2} \sum_{n \geq 1} \frac{a(n)b(n)}{n^{k + \frac{1}{2}}},
  \end{equation*}
  with $\theta = \max_j (\Im t_j)$, where $\frac{1}{2} + it_j$ correspond to the types of each form in an orthogonal basis of Maass cusp forms. (For $\SL_2(\mathbb{Z})$ we know that $\theta=0$.)
\end{theorem}

\begin{remark} Our main proof of Theorem~\ref{thm:second_moment} is for the case when $f$ and $g$ are level one. We expect the argument to completely generalize to higher level and nebentypus, and thus also to half-integral weight, but it is more difficult to explicitly verify the absence of poles of $D(s,S_f \times S_g)$ and $D(s,S_f \times \overline{S_g})$ in this case. We reserve the details concerning general level for Section~\ref{sec:genweight}.
\end{remark}

This can be thought of as a generalized, smoothed analogue of~\eqref{eq:OnAverageSquares}, as it includes a weight factor of $e^{-n/X}$ rather than a sharp cutoff function.
In particular, by specializing to eigenforms $f = g$, we recover a new proof of the Classical Conjecture on average. The disparity between the error terms, $O(X^{-\frac{1}{2}+\theta+\epsilon})$ in our theorem and $O(X^k\log^2(X))$, as cited above in \eqref{beebound}, is accounted for in the smoothing factor. Applying different Mellin transforms would allow us to obtain sharp cutoff results but by our current analysis techniques we are unable to match $O(X^k\log^2(X))$, which we still expect to be true even in the case when $f \neq g$.
Indeed, the bound on the error term $B(x)$ in \eqref{beebound} does not come from a pole of $D(s, S_f \times S_f)$, but rather from more subtle issues of convergence from using a sharp cutoff.

In Section~\ref{sec:genweight}, we show how our methods might extend to general level and demonstrate how they extend to square-free level explicitly.
Conversely, by comparing our methodology with the results of~\cite{HafnerIvic89} and~\cite{chandrasekharan1962functional}, we get Theorem~\ref{thm:general_weight_level_comparison}, which demonstrates surprising cancellation involving Kloosterman sums for general level.



\subsection*{Areas for Further Investigation}

By using either a sharp cutoff or a concentrating integral transform on the meromorphic continuation given in Theorems~\ref{thm:Wsfgmero} one could hope to extend the methods of this paper to get individual bounds on $S_f(n)^2$ instead of average bounds. For the first time, we have an analytic object with which to explore these second moments.

The greatest barrier to progress has been understanding the discrete spectrum~\eqref{line:1spectralexp}, which requires bounding
\begin{align*}
  \frac{1}{2\pi i} \int_{(\sigma)} \sum_j \rho_j(1)  &\frac{\Gamma(s - \frac{1}{2} + it_j) \Gamma(s - \frac{1}{2} - it_j)}{\Gamma(s) \Gamma(s + k - 1)}  L(s - \tfrac{1}{2}, \mu_j)  \\
&  \times  \langle \Im(\cdot)^k(f\overline{g}+T_{-1}(f\overline{g})), \mu_j \rangle V(X, s) \ ds
\end{align*}
where $V(X, s)$ is a cutoff function and $T_{-1}$ is the Hecke operator that sends $x$ in the argument of the function to $-x$.
Other choices of cutoff function should lead to more and different avenues for exploration. We suspect better understanding of bounds on terms coming from the discrete spectrum are possible with current technology, and we are eager to explore other approaches.

Our current ability to bound the discrete spectrum will be made more clear in forthcoming work, where we build on the meromorphic properties of $D(s, S_f \times S_f)$ and use the flexible cutoff function $V_Y(X,s) = \frac{X^2}{Y} \exp(\pi s^2/Y^2)$ to bound the size of $S_f(n)$ in short intervals of a particular size.

Further, bounds on the size of $\sum S_f(n)^2$ are closely related to sign changes of $S_f(n)$ and of the individual Fourier coefficients $a(n)$. In another forthcoming work, the authors investigate the sign changes in the sequence $\{S_f(n)\}$ and relations with the Classical Conjecture.

\section{Basic Tools and Notation}

In this section, we recall some basic tools and formulae.

\subsection{The Rankin-Selberg $L$-function}\label{ssec:rankin}

Here we follow the construction of $L(s,f \times g)$ given in Section 1.6 of~\cite{Bump98} but we normalize so that the functional equation corresponds to the transformation $s \to 1-s$. Let $f(z) = \sum a(n)e(nz)$ and $g(z) = \sum b(n)e(nz)$ be modular forms of weight $k$ on a congruence subgroup $\Gamma$, and where at least one is cuspidal. Let $\Gamma \backslash \mathcal{H}$ denote the upper half plane modulo the group action due to $\Gamma$ and let $\langle f,g \rangle$ denote the Petersson inner product,
\[
\langle f,g\rangle = \iint\limits_{\Gamma \backslash \mathcal{H}} f(z)\overline{g(z)} \, \frac{dxdy}{y^2}.
\]
 The Rankin-Selberg convolution $L$-function is given by the Dirichlet series expansion for $\Re s >1$,
\begin{align}\label{eq:rankinselberg}
  L(s, f\times g) &:= \zeta(2s)\sum_{n \geq 1} \frac{a(n)\overline{b(n)}}{n^{s + k - 1}} ,
  \end{align}
  which can be meromorphically continued to all $s \in \mathbb{C}$ via the identity,
  \begin{align}
  &  L(s, f\times g) = \frac{(4\pi)^{s+k-1}\zeta(2s)}{\Gamma(s + k - 1)}\left \langle \Im(\cdot)^k f\overline{g}, E(\cdot, \overline{s}) \right \rangle , \label{line:2rankinselberg}
\end{align}
where $E(z,s)$ is the real-analytic Eisenstein series
\begin{equation}\label{eq:eisenstein}
E(z,s)=  \sum_{\gamma \in \Gamma_\infty \backslash \Gamma} \Im (\gamma z)^s,
\end{equation}
for $\Re s >1$. We also note that in \eqref{line:2rankinselberg} if we replace $f\overline{g}$ with $fT_{-1}g$, where $T_{-1}$ is the Hecke operator which has action $T_{-1}F(x+iy) = F(-x+iy)$ as described in Theorem~3.12.6 of~\cite{Goldfeld2006automorphic}, we can similarly get a meromorphic continuation of $L(s,f\times \overline{g})$.
Here $\Gamma_\infty$ is the stabilizer subgroup of $\Gamma$ of the cusp at infinity. The Rankin-Selberg $L$-function is holomorphic except for, at most, a simple pole at $s =1$ when $f=g$ whose residues can be read from \eqref{line:2rankinselberg}.  When $\Gamma=\SL_2(\mathbb{Z})$, we have the functional equation
\[
  (2\pi)^{-2s }\Gamma(s)\Gamma(s + k - 1) L(s, f\times g) := \Lambda(s, f\times g) = \Lambda(1 - s, f\times g)
\]
due to the functional equation of the completed Eisenstein series $E^*(z,s)=E^*(z,1-s)$ where $E^*(z,s):=\pi^{-s}\Gamma(s)\zeta(2s)E(z,s)$. There is an analogous transformation at higher levels but its formulation is complicated by the existence of other cusps.

\subsection{Mellin-Barnes integral transform}\label{ssec:mellinbarnes}

Here we consider an integral analogue of the binomial theorem and one of the integrals considered by Barnes~\cite{Barnes}, slightly modified from the formulation presented in 6.422(3) of~\cite{GradshteynRyzhik07}.

\begin{lemma}[Barnes, 1908]
  If $0 > \gamma > -\Re ( \beta)$ and $\lvert \arg t \rvert < \pi$, then
  \[
    \frac{1}{2\pi i}\int_{\gamma - i\infty}^{\gamma + i\infty}\Gamma(-s)\Gamma(\beta + s)t^s \, ds = \Gamma(\beta)(1+t)^{-\beta}.
  \]
\end{lemma}

\subsection{Selberg spectral expansion}\label{ssec:spectraltheorem}

We review Selberg's Spectral Theorem, as presented in Theorem~15.5 of~\cite{IwaniecKowalski04}.

Let $L^2(\Gamma \backslash \mathcal{H})$ denote the space of square integrable functions with respect to the Petersson norm. Let  $f \in L^2(\Gamma \backslash \mathcal{H})$ and let $\{\mu_j(z) : j \geq 0\}$ be a complete orthonormal system for the residual and cuspidal spaces of $\Gamma \backslash \mathcal{H}$, consisting of the constant function $\mu_0(z)$ and infinitely many Maass cusp forms $\mu_j(z)$ for $j \geq 1$ with associated eigenvalues $\tfrac{1}{4} +t_j^2$ with respect to the hyperbolic Laplacian. We may assume that $\mu_j$ are also simultaneous eigenfunctions of the Hecke operators, including the $T_{-1}$ operator. Then $f(z)$ has the spectral decomposition,

\begin{equation}
  f(z) = \sum_j \langle f, \mu_j \rangle \mu_j(z) + \sum_{\mathfrak{a}} \frac{1}{4\pi}\int_\mathbb{R} \langle f, E_\mathfrak{a} (\cdot, \tfrac{1}{2} + it)\rangle E_\mathfrak{a}(z, \tfrac{1}{2} + it) \, dt,
\end{equation}
where $\mathfrak{a}$ ranges over the cusps of $\Gamma \backslash \mathcal{H}$. We casually refer to the first sum as the discrete and residual spectrum, and the sums of integrals as the continuous spectrum.

Further, these $it_j$ of the Maass cusp forms $\mu_j$ are expected to satisfy Selberg's Eigenvalue Conjecture, which says that all $t_j$ are real. It is known that $t_j$ is purely real or purely imaginary, and Selberg's Eigenvalue Conjecture has been proved for certain congruence subgroups, but not in general. We let $\theta = \sup_j\{|\Im(t_j)|\}$ denote the best known progress toward Selberg's Eigenvalue Conjecture for $\Gamma$. The current best known result for $\theta$ for all congruence subgroups is due to Kim and Sarnak, who show that $\theta \leq \frac{7}{64}$ as a consequence of a functoriality result due to Kim~\cite{KS}.

\section{Decomposition of Main Series}\label{sec:decomposition}

In this section, we decompose $D(s, S_f\times S_g)$ into smaller components that we can analyze individually.

\begin{proposition}\label{prop:decomposition}
  Let $f(z) = \sum_{n \geq 1}a(n)e(nz)$ and $g(z) = \sum_{n \geq 1}b(n)e(nz)$ be two weight $k$ cusp forms. Define $S_f(n) := \sum_{m \leq n}a(m)$ to be the partial sum of the first $n$ coefficients of $f$. Then the Dirichlet series associated to $S_f(n)\overline{S_g(n)}$ decomposes into
  \begin{align}
    D(&s, S_f \times S_g) := \sum_{n \geq 1} \frac{S_f(n)\overline{S_g(n)}}{n^{s+k-1}} \\
    &= W(s; f,g) + \frac{1}{2\pi i}\int_{(\gamma)}W(s-z; f,g) \zeta(z) \frac{\Gamma(z)\Gamma(s - z + k - 1)}{\Gamma(s + k - 1)} \ dz, \nonumber
  \end{align}
  for $1 < \gamma < \Re (s - 1)$. Here, $L(s, f\times g)$ denotes the Rankin-Selberg $L$-function as in Section~\ref{ssec:rankin}, $W(s; f,g)$ is given by
  \begin{equation}\label{eq:Wsfg}
  W(s; f,g) := \frac{L(s, f\times g)}{\zeta(2s)} + Z(s, 0, f\times g), \end{equation}
  and $Z(s, w, f\times g)$ denotes the symmetrized shifted convolution sum
  \begin{equation}
    Z(s, w, f\times g) := \sum_{n,h \geq 1} \frac{a(n)\overline{b(n-h)}+a(n-h)\overline{b(n)}}{n^{s + k - 1}h^w}.
  \end{equation}
\end{proposition}

\begin{proof}
  Expand and recollect the partial sums $S_f$ and $S_g$.
  \begin{align}
    D(s, S_f \times S_g) &= \sum_{n \geq 1}\frac{S_f(n)\overline{S_g(n)}}{n^{s+k-1}} =\sum_{n=1}^\infty \frac{1}{n^{s+k-1}} \sum_{m=1}^n a(m)\sum_{h=1}^n \overline{b(h)}.
  \end{align}
  Separate the sums over $m$ and $h$ into the cases where $m = h, m > h$, and $m < h$.  We adopt the convention that $a(n) = 0$ for $n \leq 0$ to simplify notation.
  We also reorder the sums, summing down from $n$ instead of up to $n$. With somewhat abusive summation notation, this gives
  \begin{align}
    \sum_{n=1}^\infty \frac{1}{n^{s+k-1}} \left(\sum_{h=m\geq 0} + \sum_{h>m\geq 0}+\sum_{m >h \geq 0}\right)a(n-m)\overline{b(n-h)}.
  \end{align}
  In the first above sum, we take $h = m$. In the second sum, when $h > m$, we can let $h = m + \ell$ and then sum over $m$ and $\ell$. Similarly in the third sum, when $m > h$, we can let $m = h + \ell$. This yields
  \begin{align*}
    &=\sum_{n=1}^\infty \frac{1}{n^{s+k-1}} \bigg(\sum_{m \geq 0} a(n-m)\overline{b(n-m)} \\
    &\qquad+ \sum_{\substack{\ell\geq 1 \\ m\geq 0}} a(n-m)\overline{b(n-m-\ell)} + \sum_{\substack{\ell\geq 1 \\ m\geq 0}} a(n-m-\ell)\overline{b(n-m)}\bigg).
  \end{align*}

  The cases when $m = 0$ are distinguished. They contribute
  \[
    W(s; f,g) = \frac{L(s, f\times g)}{\zeta(2s)} + Z(s, 0, f\times g).
  \]
  So $W(s; f,g)$ is a separation into diagonal, above-diagonal, and below-diagonal components. After reindexing by changing $n \mapsto n + m$, the sums over $m \geq 1$ can be rewritten as
  \begin{align*}
    \sum_{n,m \geq 1} \frac{1}{(n + m)^{s + k - 1}} \bigg(a(n)\overline{b(n)} + \sum_{\ell \geq 1} a(n)\overline{b(n - \ell)} + \sum_{\ell \geq 1} a(n - \ell)\overline{b(n)}\bigg).
  \end{align*}

By using the Mellin-Barnes transform from Section~\ref{ssec:mellinbarnes} with $t = m/n$, we decouple $m$ from $n$. Restricting to $\gamma > 1$ and $\Re s$ sufficiently large, the $m$ sum can be collected into $\zeta(z)$ and the $n$ sum can be collected into $W(s;f,g)$. Simplification completes the proof.
\end{proof}

To understand $D(s, S_f \times S_g)$, we study the analytic behaviour of $L(s, f\times g)$ and $Z(s, w, f\times g)$.
We treat $W(s; f,g)$ as a single object and we show in Section~\ref{sec:Zswfg} that the pole of $L(s, f\times g)$ will exactly cancel the rightmost pole of $Z(s, 0, f\times g)$.

\section{Analytic behaviour of $W(s; f,g)$ and $Z(s, w, f\times g)$}\label{sec:Zswfg}

For now, let $f$ and $g$ be full-integral weight holomorphic cusp forms of level one.
We expect most of our methods will generalize to arbitrary level and to half-integral weight, and we try to present the material in a way that indicates how the general methodology works.
To this end, we continue to show dependence on progress towards Selberg's Eigenvalue Conjecture to indicate how the results generalize to congruence subgroups where the conjecture is not yet verified.
We return to more general level in Section~\ref{sec:genweight}.
In this section, we produce a spectral expansion for the symmetrized shifted double Dirichlet series
\[
  Z(s, w, f\times g):=\sum_{m \geq 1} \sum_{\ell \geq 1} \frac{a(m)\overline{b(m-\ell)}+a(m-\ell)\overline{b(m)}}{m^{s+k-1}\ell^w},
\]
and use it to understand the analytic behaviour of $W(s; f,g)$.

\subsection{Spectral Expansion}\label{ssec:spectral}

For an integer $h \geq 1$, define the weight zero Poincar\'{e} series on $\Gamma$,
\[
  P_h(z,s):=\sum_{\gamma \in\Gamma_\infty \backslash \Gamma} \Im(\gamma z)^s e\left(h \gamma z\right),
\]
defined initially for $\Re(s)$ sufficiently positive and with meromorphic continuation to all $s\in\mathbb{C}$.

Let $\mathcal{V}_{f,g}(z) :=y^k (f \overline{g}+T_{-1}(f\overline{g}))$, which is a function in $L^2(\Gamma \backslash \mathbb{H})$. By expanding the Petersson inner product below we get
\begin{align*}
  \langle \mathcal{V}_{f,g} ,P_h(\cdot, \overline{s}) \rangle &=\frac{\Gamma\left({s}+k -1\right)}{\left(4\pi \right)^{{s}+ k -1}} D_{f,g}({s};h),
\end{align*}
where we mirror the notation in~\cite{HoffsteinHulse13} and define
\[
  D_{f,g}(s;h) := \sum_{n \geq 1}\frac{a(n)\overline{b(n-h)}+a(n-h)\overline{b(n)}}{n^{s+k-1}},
\]
again for $\Re(s)$ sufficiently positive.
Dividing by $h^w$ and summing over $h \geq 1$ recovers $Z(s, w, f\times g)$,
\begin{equation}\label{LHS}
  Z(s, w, f\times g) := \sum_{n, h \geq 1} \frac{D_{f, g}(s;h)}{h^w} = \frac{(4\pi)^{s+k-1}}{\Gamma(s+k-1)} \sum_{h \geq 1} \frac{\langle \mathcal{V}_{f,g}, P_h \rangle}{h^w},
\end{equation}
when both $\Re(s)$ and $\Re(w)$ are sufficiently positive.

We will obtain a meromorphic continuation of $Z(s, w, f\times g)$ by using the spectral expansion of the Poincar\'e series and substituting it into \eqref{LHS}. Let $\{\mu_j\}$ be an orthonormal basis of Maass eigenforms with associated types $\frac{1}{2} + it_j$ for $L^2(\Gamma \backslash \mathcal{H})$ as in Section~\ref{ssec:spectraltheorem}, each with Fourier expansion
\[
  \mu_j(z)=\sum_{n \neq 0} \rho_j(n)y^{\frac{1}{2}}K_{it_j}(2\pi \vert n \vert y)e^{2\pi i n x}.
\]
Then the inner product of $\mu_j$ against the Poincar\'{e} series gives
\begin{align}\label{eq:mujP}
  \langle P_h(\cdot,s),\mu_j \rangle &= \frac{\overline{\rho_j(h)}\sqrt{\pi}}{(4\pi h)^{s-\frac{1}{2}}} \frac{\Gamma(s-\frac{1}{2}+it_j)\Gamma(s-\frac{1}{2}-it_j)}{\Gamma(s)}.
\end{align}
\begin{remark}
  In the computation of this inner product and the inner product of the Eisenstein series against the Poincar\'e series, we use formula~\cite[$\S$6.621(3)]{GradshteynRyzhik07} to evaluate the final integrals.
\end{remark}

%

Let $E(z,w)$ be the Eisenstein series on $\SL_2(\mathbb{Z})$ as in \eqref{eq:eisenstein}. Then $E(z,w)$ has Fourier expansion (as in~\cite[Chapter 3]{Goldfeld2006automorphic}) given by
\begin{align}
  E(z,w)& =y^w + \phi(w)y^{1-w}  \\
 &\quad +  \frac{2\pi^w \sqrt{y} }{\Gamma(w)\zeta(2w)}\sum_{m \neq 0} \vert m \vert^{w-\frac{1}{2}}\sigma_{1-2w}(\vert m \vert) K_{w-\frac{1}{2}}(2\pi \vert m \vert y)e^{2\pi i m x}, \notag
\end{align}
where
\[
  \phi(w) = \sqrt{\pi} \frac{\Gamma(w - \tfrac{1}{2})\zeta(2w - 1)}{\Gamma(w)\zeta(2w)}.
\]
The inner product of the Poincar\'{e} series ($h \geq 1$) against the Eisenstein series $E(z,w)$ is given by
\begin{align}\label{eq:PhE}
  \left\langle P_h(\cdot,s),E(\cdot,w)\right\rangle =\frac{2\pi^{\overline{w}+\frac{1}{2}} h^{\overline{w}-\frac{1}{2}}\sigma_{1-2\overline{w}}(h)}{\zeta(2\overline{w})(4\pi h)^{s-\frac{1}{2}}}\frac{\Gamma(s+\overline{w}-1)\Gamma(s-\overline{w})}{\Gamma(\overline{w})\Gamma(s)},
\end{align}
provided that $\mathrm{Re}\,s >\frac{1}{2}+\vert \mathrm{Re}\,w-\frac{1}{2}\vert$. For $t$ real, $w=\frac{1}{2}+it$, and $\mathrm{Re}\,s>\frac{1}{2}$, \eqref{eq:PhE} specializes to
\begin{equation}\label{eq:PhEspecialized}
  \langle P_h(\cdot,s),E(\cdot,\tfrac{1}{2}+it)\rangle=\frac{2\sqrt{\pi} \sigma_{2it}(h)}{\Gamma(s)(4\pi h)^{s-\frac{1}{2}}}\frac{\Gamma(s-\frac{1}{2}+it)\Gamma(s-\frac{1}{2}-it)}{h^{it}\zeta^*(1-2it)},
\end{equation}
in which $\zeta^*(2s):=\pi^{-s}\Gamma(s)\zeta(2s)$ denotes the completed zeta function.

The spectral expansion of the Poincar\'{e} series is given by
\begin{align}\label{eq:Pspectral}
  \begin{split}
    P_h(z,s)&=\sum_j \langle P_h(\cdot,s),\mu_j \rangle \mu_j(z) \\
            &\quad + \frac{1}{4\pi}\int_{-\infty}^\infty\langle P_h(\cdot,s),E(\cdot,\tfrac{1}{2}+it)\rangle E(z,\tfrac{1}{2}+it)\,dt.
  \end{split}
\end{align}
We shall refer to the above sum and integral as the discrete and continuous spectrum, respectively. After substituting \eqref{eq:mujP} into the discrete part of \eqref{eq:Pspectral}, the discrete spectrum takes the form
\[
  \frac{\sqrt{\pi}}{(4\pi h)^{s-\frac{1}{2}}\Gamma(s)}\sum_j \overline{\rho_j(h)} \Gamma(s-\tfrac{1}{2}+it_j)\Gamma(s-\tfrac{1}{2}-it_j) \mu_j(z)
\]
and is analytic in $s$ in the right half-plane $\mathrm{Re}\, s> \frac{1}{2}+\theta$, where $\theta = \sup_j\{\Im(t_j)\} \leq \frac{7}{64}$ is the best known progress toward Selberg's Eigenvalue Conjecture. After inserting \eqref{eq:PhEspecialized}, the continuous spectrum takes the form
\begin{align*}
  \frac{\sqrt{\pi}}{2\pi(4\pi h)^{s-\frac{1}{2}}}\int_{-\infty}^\infty \frac{\sigma_{2it}(h)}{h^{it}}\frac{\Gamma(s-\frac{1}{2}+it)\Gamma(s-\frac{1}{2}-it)}{\zeta^*(1-2it)\Gamma(s)}E(z,\tfrac{1}{2}+it)\,dt,
\end{align*}
which has its right-most poles in $s$ when $\mathrm{Re}\,s=\frac{1}{2}$.

Substituting this spectral expansion into~\eqref{LHS} and executing the sum over $h \geq 1$ gives the following proposition.

\begin{proposition}\label{prop:spectralexpansionfull}
  For $f,g$ weight $k$ forms on $\SL_2(\mathbb{Z})$, the shifted convolution sum $Z(s, w, f\times g)$ can be expressed as
\begin{align}
  Z(s&, w, f\times g) := \sum_{m=1}^\infty \frac{a(m) \overline{b(m-h)}+a(m-h)\overline{b(m)} }{m^{s+k -1}h^w} \nonumber \\
     &= \frac{(4\pi )^k}{2} \sum_j\rho_j(1) G(s, i t_j) L(s + w -\tfrac{1}{2},\mu_j)\langle \mathcal{V}_{f,g},\mu_j \rangle \label{line:1spectralexp} \\
     &\quad+\frac{(4\pi)^{k}}{4\pi i}\int_{(0)} G(s, z) \mathcal{Z}(s,w,z) \langle \mathcal{V}_{f,g},E(\cdot,\tfrac{1}{2}-\overline{z})\rangle \,dz, \label{line:2spectralexp}
\end{align}
when $\Re (s+w)>\frac{3}{2}$, where $G(s, z)$ and $\mathcal{Z} (s,w,z)$ are the collected $\Gamma$ and $\zeta$ factors of the discrete and continuous spectra,
\begin{align*}
  G(s, z) &= \frac{\Gamma(s - \tfrac{1}{2} + z)\Gamma(s - \tfrac{1}{2} - z)}{\Gamma(s)\Gamma(s+k-1)} \\
  \mathcal{Z}(s,w,z) &= \frac{\zeta(s + w -\frac{1}{2} + z)\zeta(s + w -\frac{1}{2} - z)}{\zeta^*(1+2z)}.
\end{align*}
\end{proposition}


\begin{remark}\label{extraremark}
Consider Stirling's approximation, that for $x,y\in\mathbb{R}$,
\[
\Gamma(x+iy) \sim (1+|y|)^{x-\frac{1}{2}}e^{-\frac{\pi}{2}|y|}
\]
as $y \to \pm\infty$ with $x$ bounded. Thus, we have that for vertical strips in $s$ and $z$,
\[
G(s,z) \sim P(s,z) e^{-\frac{\pi}{2}(2\max(|s|,|z|)-2|s|)},
\]
where $P(s,z)$ has at most polynomial growth in $s$ and $z$. By Watson's triple product formula in the integral-weight case, given in Theorem 3 of~\cite{watson2008rankin}, and K\i{}ral's bound should we require  the half-integral weight case, given in Proposition 13 of \cite{mehmet}, we know that
\[
\rho_j(1)\langle f\overline{g} \Im(\cdot)^k,\overline{\mu_j}\rangle \ \ \mbox{ and } \ \ \rho_j(1)\langle T_{-1}(f\overline{g}) \Im(\cdot)^k,\overline{\mu_j}\rangle
\] has at most polynomial growth in $|t_j|$. The same can be said about
\[
\langle \mathcal{V}_{f,g},E(\cdot,\tfrac{1}{2}+z)\rangle / \zeta^*(1+2z),
\] albeit through more direct computation. From this it is clear that \eqref{line:1spectralexp} and \eqref{line:2spectralexp} converge uniformly on vertical strips in $t_j$ and have at most polynomial growth in $s$.
\end{remark}

\subsection{Meromorphic Continuation}\label{ssec:meromorphic_continuation}

In this section, we seek to understand the meromorphic continuation and polar behaviour of $Z(s, 0, f\times g)$. This naturally breaks down into two parts: the contribution from the discrete spectrum and the contribution from the continuous spectrum.

\subsubsection{The Discrete Spectrum.}

Examination of line~\eqref{line:1spectralexp}, the contribution from the discrete spectrum, reveals that the poles come only from $G(s, it_j)$.
There are apparent poles when $s = \tfrac{1}{2} \pm it_j - n$ for $n \in \mathbb{Z}_{\geq 0}$.
Interestingly, the first set of apparent poles are at $s = \frac{1}{2} \pm it_j$ do not actually occur.

\begin{lemma}\label{lem:Litj_equals_zero}
  For even Maass forms $\mu_j$, we have $L(\pm it_j, \mu_j) = 0$.
\end{lemma}

\begin{proof}
  The completed $L$-function associated to a Maass form $\mu_j$ is given by
  \begin{equation} \label{eq:feq}
    \Lambda_j(s) = \pi^{-s} \Gamma\left( \tfrac{s + \epsilon + it_j}{2} \right)\Gamma\left( \tfrac{s + \epsilon - it_j}{2} \right) L(s, \mu_j) = (-1)^\epsilon \Lambda_j(1-s),
  \end{equation}
  as in~\cite[Sec 3.13]{Goldfeld2006automorphic}, where $\epsilon = 0$ if the Maass form $\mu_j$ is even and $1$ if it is odd. As the completed $L$-function is entire, $L(\pm it_j, \mu_j)$ is a trivial zero.
\end{proof}

Similarly, $L(-2n \pm it_j, \mu_j), n \! \in \!\mathbb{Z}_{\geq 0}$ are trivial zeroes for even Maass forms.

\begin{lemma}\label{lem:oddorthogonaltoeven}
  Suppose $f$ and $g$ are weight $k$ cusp forms, as above. For odd Maass forms $\mu_j$, we have $\langle \mathcal{V}_{f,g}, \mu_j \rangle = 0$.
\end{lemma}

\begin{proof}
  Recall that $\mathcal{V}_{f,g} :=y^k(f\overline{g}+T_{-1}(f\overline{g}))$, so clearly $T_{-1}\mathcal{V}_{f,g}=\mathcal{V}_{f,g}$. Since $T_{-1}$ is a self-adjoint operator with respect to the Petersson inner product we have that
  \[
  \langle \mathcal{V}_{f,g},\mu_j \rangle =  \langle T_{-1}\mathcal{V}_{f,g},\mu_j \rangle  =   \langle \mathcal{V}_{f,g},T_{-1}\mu_j \rangle = -  \langle \mathcal{V}_{f,g},\mu_j \rangle.
  \]
  Thus $  \langle \mathcal{V}_{f,g},\mu_j \rangle  =0$.
\end{proof}

Lemmas~\ref{lem:oddorthogonaltoeven} guarantees that the only Maass forms appearing in line~\eqref{line:1spectralexp} are even. The first set of apparent poles from even Maass forms appear at $s = \frac{1}{2} \pm it_j$ and occur as simple poles of the gamma functions in the numerator of $G(s, t_j)$. They come multiplied by the value of $L(it_j, \mu_j)$, which by Lemma~\ref{lem:Litj_equals_zero} is zero.

%

So, in summary, $D(s, S_f \times S_g)$ has no poles at $s = \frac{1}{2} \pm it_j$. The next set of apparent poles are at $s = -\frac{1}{2} \pm it_j$, appearing at the next set of simple poles of the gamma functions in the numerator. Unlike the previous poles, these do not occur at trivial zeroes of the $L$-function. So we have poles of the discrete spectrum at $s = -\frac{1}{2} \pm it_j$.


\subsubsection{The Continuous Spectrum.}\label{conspec}

Let us now examine line~\eqref{line:2spectralexp}, the contribution from the continuous spectrum. This is substantially more involved than the discrete spectrum and exhibits remarkable cancellation.

The rightmost pole seems to occur from the pair of zeta functions in the numerator, occurring when $s + w - \frac{1}{2} \pm z = 1$. We must disentangle $s$ and $w$ from $z$ in order to understand these poles.

Line~\eqref{line:2spectralexp} is analytic for $\Re (s+w) > \frac{3}{2}$. For $s$ with $\Re s \in (\frac{3}{2} - \Re w, \frac{3}{2} - \Re w + \epsilon)$ for some very small $\epsilon$, we want to shift the contour of integration, avoiding poles coming from the $\zeta^*(1 -2z)$ appearing in the denominator of the expansion of $E(\cdot,\frac{1}{2}+\overline{z})$. So we shift the $z$-contour to the right while staying within the zero-free region of $\zeta$. By an abuse of notation, we denote this shift here by $\Re z = \epsilon$ and let $\epsilon$ in this context actually refer to the real value of the $z$-contour at the relevant imaginary value. This argument can be made completely rigourous, cf.~\cite[p. 481-483]{HoffsteinHulse13}. By the residue theorem,
\begin{align}\label{eq:FEfirstpart}
  &\frac{(4\pi)^k}{4\pi i} \int_{(0)}G(s,z) \mathcal{Z}(s,w,z) \langle \mathcal{V}_{f,g}, E\rangle  \ dz \\
  &= \frac{(4\pi)^k}{4\pi i}\int_{(\epsilon)} G \mathcal{Z} \langle \mathcal{V}_{f,g}, E \rangle \  dz - \frac{(4\pi)^k}{2}\Res_{z = s + w - \frac{3}{2}} G \mathcal{Z} \langle \mathcal{V}_{f,g}, E\rangle \nonumber,
\end{align}
where the above residue is found to be
\begin{equation}\label{eq:FEfirstresidue}
 -\frac{\zeta(2s + 2w - 2)\Gamma(2s + w - 2)\Gamma(1-w)}{\zeta^*(2s+2w-2)\Gamma(s)\Gamma(s + k - 1)}\langle \mathcal{V}_{f,g}, E(\cdot, 2-\overline{s}-\overline{w})\rangle.
\end{equation}
The residue is analytic in $s$ for $\Re s \in (1 - \Re w, \tfrac{3}{2} - \Re w + \epsilon)$, and has an easily understood meromorphic continuation to the whole plane. Notice also that the shifted contour integral has no poles in $s$ for $\Re s \in (\tfrac{3}{2} - \Re w - \epsilon, \tfrac{3}{2} - \Re w + \epsilon)$, so we have found an analytic (not meromorphic!) continuation in $s$ of Line~\eqref{line:2spectralexp} past the first apparent pole along $\Re s = \frac{3}{2} - \Re w$.

For $s$ with $\Re s \in (\frac{3}{2} - \Re w - \epsilon, \frac{3}{2} - \Re w)$, we shift the contour of integration back to $\Re z = 0$.
Since this passes a pole, we pick up a residue.
But notice that this is the residue at the \emph{other} pole,
\begin{align}\label{eq:FEsecondpart}
  &\frac{(4\pi)^k}{4\pi i} \int_{(\epsilon)}G(s, w, z) \mathcal{Z}(s,w,z) \langle \mathcal{V}_{f,g}, E\rangle dz  \\
  =&\frac{(4\pi)^k}{4\pi i}\int_{(0)} G \mathcal{Z} \langle \mathcal{V}_{f,g}, E \rangle dz + \frac{(4\pi)^k}{2}\Res_{z = \frac{3}{2} - s - w} G \mathcal{Z}  \langle \mathcal{V}_{f,g}, E\rangle \nonumber.
\end{align}
By using the functional equations of the Eisenstein series and zeta functions, one can check that
\[
  \Res_{z = \frac{3}{2} - s - w} G \mathcal{Z} \langle \mathcal{V}_{f,g}, E\rangle = - \Res_{z = s + w -  \frac{3}{2}} G \mathcal{Z} \langle\mathcal{V}_{f,g}, E\rangle.
\]
So \eqref{line:2spectralexp}, originally defined for $\Re s > \frac{3}{2} - \Re w$, has meromorphic continuation for $\frac{1}{2} - \Re w < \Re s < \frac{3}{2} - \Re w$ given by
\begin{equation}
  \frac{(4\pi)^k}{4\pi i}\int_{(0)}G \mathcal{Z} \langle \mathcal{V}_{f,g}, E\rangle dz + (4\pi)^k \!\!\!\!\Res_{z = \frac{3}{2} - s - w} \!\!\!\! G \mathcal{Z} \langle \mathcal{V}_{f,g}, E\rangle.
\end{equation}

A very similar argument works to extend the meromorphic continuation in $s$ of the contour integral past the next apparent pole at $\Re s  = \frac{1}{2}$, leading to a meromorphic continuation in the region $-\frac{1}{2}< \Re s < \frac{1}{2}$ given by
\begin{align}\label{eq:fullcontinuation}
  &\frac{(4\pi)^k}{4\pi i}\int_{(0)}G(s, w, z)\mathcal{Z}(s,w,z)\langle \mathcal{V}_{f,g}, E(\cdot, \tfrac{1}{2} - \overline{z})\rangle dz \\
  &\quad+ (4\pi)^k \!\!\!\!\Res_{z = \frac{3}{2} - s - w} \!\!\!\! G(s, w, z) \mathcal{Z}(s,w,z) \langle \mathcal{V}_{f,g}, E(\cdot, \tfrac{1}{2} - \overline{z})\rangle \label{line:firstresidual} \\
  &\quad+ (4\pi)^k \!\!\Res_{z = \frac{1}{2} - s} \!\! G(s, w, z) \mathcal{Z}(s,w,z) \langle \mathcal{V}_{f,g}, E(\cdot, \tfrac{1}{2} -\overline{z})\rangle. \label{line:secondresidual}
\end{align}

We iterate this argument, as in Section~4 of~\cite[p. 481-483]{HoffsteinHulse13}. Somewhat more specifically, when $\Re(s)$ approaches a negative half-integer, $\frac{1}{2}-n$, we can shift the line of integration for $z$ right past the pole due to $G(s,z)$ at $z=s-\frac{1}{2}+n$, move $s$ left past the line $\Re(s)=\frac{1}{2}-n$ and then shift the line of integration for $z$ left, back to zero and over the pole at $z=\frac{1}{2}-s-n$. This gives meromorphic continuation of~\eqref{line:2spectralexp} arbitrarily far to the left, accumulating a pair of residual terms at each half-integer line as is the case in \eqref{line:secondresidual}.

We now specialize to $w = 0$. The rightmost pole of~\eqref{line:2spectralexp} occurs in the first residual term appearing in~\eqref{line:firstresidual} from the meromorphic continuation. The pole occurs at $s = 1$ from the Eisenstein series and has residue
\begin{align}\label{eq:Zresidue}
  &\Res_{s = 1} \Res_{z = s - \frac{3}{2}} (4\pi)^k G(s, 0, z) \mathcal{Z}(s,0,z) \langle \mathcal{V}_{f,g}, E(\cdot, \tfrac{1}{2} - \overline{z})\rangle  = \nonumber \\
  &= \Res_{s = 1} \frac{(4\pi)^k \zeta(2s - 2)\Gamma(2s - 2)}{\zeta^*(2s-2)\Gamma(s)\Gamma(s + k - 1)}\langle \mathcal{V}_{f,g}, E(\cdot, 2-\overline{s})\rangle,
\end{align}
which can be interpreted as (see Section~\ref{ssec:rankin})
\begin{align}
  -\frac{(4\pi)^k}{\, \Gamma(k)}\frac{3}{\pi} \langle f\Im(\cdot)^k, g \rangle = - \Res_{s = 1}\frac{L(s, f\times g)}{\zeta(2)}.
\end{align}

The next pole of~\eqref{line:2spectralexp} also occurs in the first residual term appearing in~\eqref{line:firstresidual}, occurring at $s = \frac{1}{2}$ from the gamma function in the numerator of $G(s, 0, z)$. Otherwise, the continuous spectrum~\eqref{line:2spectralexp} is analytic for $\Re s \geq \frac{1}{2}$. Combining this continuation with the continuation of the discrete spectrum~\eqref{line:1spectralexp}, we get the following lemma.

\begin{lemma}\label{lemma:Zswfg}
  Maintaining the notation from Proposition~\ref{prop:spectralexpansionfull}, $Z(s, 0, f\times g)$ has meromorphic continuation in $s$ to $\Re s \geq \frac{1}{2}$ with poles at most at $s= 1$ and $s = \frac{1}{2}$.
  The rightmost pole is at $s = 1$ and has residue
  \begin{align}\label{eq:Zswfg_poleatone}
    -\frac{(4\pi)^k}{\, \Gamma(k)}\frac{3}{\pi} \langle f\Im(\cdot)^k, g \rangle = - \Res_{s = 1}\frac{L(s, f\times g)}{\zeta(2)}.
  \end{align}
\end{lemma}

Returning to the meromorphic continuation of~\eqref{line:2spectralexp} given just above, we evaluate the second residual term~\eqref{line:secondresidual}, which only appears for $\Re s < \frac{1}{2}$,
\begin{align}\label{eq:secondresidual}
  &\Res_{z = \frac{1}{2} - s} (4\pi)^k G(s, 0,  z)\mathcal{Z}(s, 0, z) \langle \mathcal{V}_{f,g}, E(\cdot, \tfrac{1}{2} - \overline{z}) \rangle \\
  &= (4\pi)^k \frac{\zeta(0)\zeta(2s - 1)\Gamma(2s - 1)}{\zeta^*(2-2s)\Gamma(s)\Gamma(s + k - 1)}\langle \mathcal{V}_{f,g}, E(\cdot, \overline{s}) \rangle. \label{eq:secondresidual_example}
\end{align}

By using the gamma duplication formula
\[
  \frac{\Gamma(2s - 1)}{\Gamma(s)} = \Gamma(s - \tfrac{1}{2})\frac{2^{2s-2}}{\sqrt \pi}
\]
and functional equations for $\zeta(s)$ and $E(z,s)$, we can rewrite~\eqref{eq:secondresidual} as
\begin{equation}\label{eq:secondresidualsimple}
  - \frac{L(s, f\times g)}{\zeta(2s)}.
\end{equation}
Thus this second residual term has poles at zeroes of $\zeta(2s)$.

More generally, a residual term
\begin{equation}\label{eq:generalresidual}
  \frac{(-1)^j(4\pi)^k}{\Gamma(j+1)} \frac{\zeta(-j)\zeta(2s+j-1)\Gamma(2s+j-1)}{\zeta^*(2-2s-2j) \Gamma(s)\Gamma(s+k-1) } \langle \mathcal{V}_{f,g}, E(\cdot,\overline{s+j})\rangle
\end{equation}
is introduced for $\Re s < \frac{1}{2} - j$ in the continuation past the apparent polar line $\Re s = \frac{1}{2} - j$ for each integer $j \geq 0$. We recognize~\eqref{eq:secondresidual_example} as the $j = 0$ case of~\eqref{eq:generalresidual}. We note that the first residual term~\eqref{line:firstresidual} is distinguished in coming from a pole from the zeta function, while all further residual terms have the same form as~\eqref{eq:generalresidual} and come from poles from gamma functions. As in~\eqref{eq:secondresidualsimple}, the Eisenstein series appearing in the $j$th residual~\eqref{eq:generalresidual} introduces poles at $s = \frac{\rho}{2} - j$ for each nontrivial zero $\rho$ of $\zeta(s)$.


\subsection{Analytic Behaviour of $W(s; f,g)$}

Recall that
\[
  W(s; f,g) = \frac{L(s, f\times g)}{\zeta(2s)} + Z(s, 0, f\times g).
\]
Then Lemma~\ref{lemma:Zswfg} shows that the leading pole of $\frac{L(s, f\times g)}{\zeta(2s)}$ at $s = 1$ cancels perfectly with the leading pole of $Z(s, 0, f\times g)$. So $W(s; f,g)$ is analytic for $\Re s > \tfrac{1}{2}$ and has a pole at $s = \tfrac{1}{2}$. Further, since the meromorphic continuation of $Z(s, 0, f\times g)$ has~\eqref{eq:secondresidualsimple} as a residual for $\Re s < \tfrac{1}{2}$, we see that the Rankin-Selberg $L$-function $\frac{L(s,f\times g)}{\zeta(2s)}$ term completely cancels when $\Re s < \tfrac{1}{2}$. Then $W(s; f,g)$ has a meromorphic continuation to $\mathbb{C}$, and for $\Re s > -\tfrac{1}{2}$ the only possible poles are $s = \tfrac{1}{2}$, coming from the first residual term~\eqref{line:firstresidual} of the continuous spectrum~\eqref{line:2spectralexp} and those at $s = -\tfrac{1}{2} \pm it_j$, coming from the exceptional eigenvalues of the discrete spectrum~\eqref{line:1spectralexp}. (There are no exceptional eigenvalues for $\Gamma = \SL_2(\mathbb{Z})$).

Let us evaluate the residue of $Z(s, 0, f\times g)$ at the pole $s = \tfrac{1}{2}$. For ease, we write the first residual term as a residue at $z = \tfrac{3}{2} - s$,
\begin{equation}
  (4\pi)^k \frac{\zeta(2s - 2)\Gamma(2s - 2)}{\zeta^*(2s - 2)\Gamma(s)\Gamma(s + k - 1)}\langle \mathcal{V}_{f,g}, E(\cdot, 2 - \overline{s}) \rangle.
\end{equation}
By applying the gamma duplication formula, expanding the completed zeta function in the denominator and cancelling similar terms from the numerator and denominator, this becomes
\begin{equation}\label{eq:later_shortcut}
  \frac{(4\pi)^{s + k - 1}}{2\sqrt \pi} \frac{\Gamma(s - \frac{1}{2})}{\Gamma(s)\Gamma(s + k - 1)}\langle \mathcal{V}_{f,g}, E(\cdot, 2 - \overline{s}) \rangle.
\end{equation}
There is a pole at $s = \frac{1}{2}$ coming from the gamma function in the numerator. The residue at this pole is given by
\begin{equation}
  \frac{1}{2\sqrt \pi \Gamma(\frac{1}{2})}\frac{(4\pi)^{k - \frac{1}{2}}}{\Gamma(k - \frac{1}{2})}\langle \mathcal{V}_{f,g}, E(\cdot, \tfrac{3}{2}) \rangle.
\end{equation}
We rewrite this as a special value of the Rankin-Selberg $L$-function
\begin{equation}
  \frac{1}{2 \pi}\frac{(k - \frac{1}{2})}{4\pi} \frac{(4\pi)^{k + \frac{1}{2}}}{\Gamma(k + \frac{1}{2})}\langle \mathcal{V}_{f,g}, E(\cdot, \tfrac{3}{2}) \rangle = \frac{(k - \frac{1}{2})}{4\pi^2} \frac{L(\frac{3}{2}, f\times g)}{\zeta(3)}.
\end{equation}
This allows us to conclude the following theorem.

\begin{theorem}\label{thm:Wsfgmero}
  Let $f,g$ be two holomorphic cusp forms on $\SL_2(\mathbb{Z})$. Maintaining the same notation as above, the function $W(s; f,g)$ has a meromorphic continuation to $\mathbb{C}$ given by~\eqref{line:2rankinselberg} and Proposition~\ref{prop:spectralexpansionfull} with potential poles at $s$ with $\Re s \leq \tfrac{1}{2}$ and $s \in \mathbb{Z}\cup(\mathbb{Z} + \tfrac{1}{2})\cup\mathfrak{S}\cup\mathfrak{Z}$, where  $\mathfrak{Z}$ denotes the set of shifted zeta-zeroes $\{-1 + \frac{\rho}{2} - n: n \in \mathbb{Z}_{\geq 0}\}$, and $\mathfrak{S}$ denotes the set of shifted discrete types $\{-\tfrac{1}{2} \pm it_j - n: n \in \mathbb{Z}_{\geq 0}\}$.

The leading pole is at $s = \frac{1}{2}$ and
  \begin{equation}
    \Res_{s = \frac{1}{2}} W(s; f, g) = \frac{(k - \frac{1}{2})}{4\pi^2} \frac{L(\tfrac{3}{2}, f\times g)}{\zeta(3)}.
  \end{equation}
\end{theorem}

With this theorem and the decomposition from Proposition~\ref{prop:decomposition}, we have the meromorphic continuation of the Dirichlet series $D(s, S_f \times S_f)$.

\begin{remark}Very similar work gives the meromorphic continuation for $D(s, S_f \times \overline{S_g})$, mainly replacing $\overline{g}$ with $T_{-1}g$ in the above formulation. This distinction only matters at higher levels when $f$ and $g$ have nontrivial nebentypus, and the spectral expansion is modified accordingly. \end{remark}

\section{Second Moment of Sums of Fourier Coefficients}\label{sec:second_moment_of_sums}

Now we will use the results of the previous section to prove Theorem~\ref{thm:second_moment} for cusp forms on $\SL_2(\mathbb{Z})$ and suggest how the argument generalizes.
First, we state a simple corollary of Theorem~\ref{thm:Wsfgmero}.

\begin{corollary}\label{lem:Wsfg_analytic}
  Let $\theta = \max_j {\Im(t_j)} \leq \frac{7}{64}$ denote the progress towards Selberg's Eigenvalue Conjecture for the given congruence subgroup $\Gamma$.
  The function
  \[
    W(s; f, g) = \frac{L(s, f\times g)}{\zeta(2s)} + Z(s, 0, f\times g),
  \]
  appearing in Proposition~\ref{prop:decomposition}, is analytic for $\Re s > -\frac{1}{2} + \theta$ except for a simple pole at $s = \frac{1}{2}$.
\end{corollary}

We now consider a smooth cutoff integral of $D(s, S_f \times S_g)$. Using the well-known integral transform,
\begin{equation}\label{eq:smoothcutofftransform}
  \frac{1}{2\pi i}\int_{(\sigma)}D(s, S_f \times S_g) X^s \Gamma(s)ds = \sum_{n \geq 1} \frac{S_f(n) \overline{S_g(n)} }{n^{k - 1}}e^{-n/X},
\end{equation}
for $\sigma$ large enough to be in the domain of absolute convergence of $D(s, S_f\times S_g)$, say  $\sigma = 4$. To understand the right hand side of~\eqref{eq:smoothcutofftransform}, we decompose the left hand side as in Proposition~\ref{prop:decomposition}. We thus investigate the two integrals,
\begin{equation}\label{eq:smoothintegral1}
  \frac{1}{2\pi i}\int_{(4)} W(s; f,g) X^s \Gamma(s)ds
\end{equation}
and
\begin{equation}\label{eq:smoothintegral2}
  \frac{1}{(2\pi i)^2}\int_{(4)}\int_{(2)} W(s-z; f,g)\zeta(z)\frac{\Gamma(z)\Gamma(s - z + k - 1)}{\Gamma(s+k-1)}\, dz \, X^s  \Gamma(s) \, ds.
\end{equation}

\begin{lemma}\label{lem:Wsfg_regular}
  Fix an $\epsilon > 0$. Then the integral~\eqref{eq:smoothintegral1} is
  \begin{align}\label{eqsplit}
    \begin{split}
    & \frac{(k - \tfrac{1}{2})}{4\pi^2} \frac{L(\frac{3}{2}, f\times g)}{\zeta(3)} \Gamma(\tfrac{1}{2}) X^{1/2} + O_\epsilon(X^\epsilon),
    \end{split}
  \end{align}
  where $O_\epsilon(\cdot)$ indicates that the implicit constant depends on $\epsilon$.
\end{lemma}


\begin{proof}
  Shifting the line of integration to $\Re s = \epsilon$ passes the pole at $s=\frac{1}{2}$m with residue as given by Theorem~\ref{thm:Wsfgmero}. To bound the shifted integral, we observe that $W(s; f,g)$ has at most polynomial growth in vertical strips while $\Gamma(s)$ has exponential decay.
\end{proof}

\begin{lemma}\label{lem:Wsfg_MellinBarnes}
  Fix an $\epsilon > 0$. Then the integral~\eqref{eq:smoothintegral2} is
  \begin{align}
    \begin{split}
    & \frac{1}{4\pi^2} \frac{L(\frac{3}{2}, f\times g)}{\zeta(3)} \Gamma(\tfrac{3}{2}) X^{3/2} + O_\epsilon(X^{\frac{1}{2} + \theta + \epsilon}).
    \end{split}
  \end{align}
\end{lemma}

\begin{proof}

  We first shift the $z$ line of integration to $\epsilon$, passing a pole at $z = 1$ from $\zeta(z)$ with residue
  \begin{equation}
    \frac{1}{2\pi i}\int_{(4)}W(s - 1; f,g)\frac{1}{s + k -2}X^s \Gamma(s)\, ds.
  \end{equation}
  The remaining analysis of this integral is almost identical to the analysis of~\eqref{eq:smoothintegral1}. We shift the line of integration to $\Re s = \frac{1}{2} + \theta + \epsilon$, passing over the pole at $s=\frac{1}{2}$. The integrand has exponential decay in vertical strips, so the $s$-shifted integral is $O_\epsilon(X^{\frac{1}{2} + \theta + \epsilon})$.

  All that remains is the shifted double integral
  \begin{equation}
    \frac{1}{(2\pi i)^2}\int_{(4)}\int_{(\epsilon)} W(s-z; f,g)\zeta(z)\frac{\Gamma(z)\Gamma(s - z + k - 1)}{\Gamma(s+k-1)} \, dz \, X^s \Gamma(s) \, ds.
  \end{equation}
  We shift the line of $s$ integration to $\Re s = \tfrac{1}{2} + 2\epsilon$ without encountering any poles. Again we have exponential decay in vertical strips in both $s$ and $z$, so we can conclude that this integral is also $O_\epsilon(X^{\tfrac{1}{2} + 2\epsilon})$. Putting these together gives the lemma.

\end{proof}

By combining Lemmas~\ref{lem:Wsfg_regular} and~\ref{lem:Wsfg_MellinBarnes}, we have proved Theorem~\ref{thm:second_moment} for $D(s, S_f \times S_g)$, as stated in the introduction, for level one.

\begin{theorem*}
  If $f$ and $g$ are holomorphic cusp forms for $\SL_2\mathbb(\mathbb{Z})$, then for any $\epsilon > 0$,
  \begin{equation*}
    \frac{1}{X} \sum_{n \geq 1}\frac{S_f(n)\overline{S_g(n)}}{n^{k - 1}}e^{-n/X} = CX^{\frac{1}{2}} + O_{f,g,\epsilon}(X^{-\frac{1}{2} + \theta + \epsilon})
  \end{equation*}
  where
  \begin{equation*}
    C = \frac{\Gamma(\tfrac{3}{2})}{4\pi^2} \frac{L(\frac{3}{2}, f\times g)}{\zeta(3)}  = \frac{\Gamma(\tfrac{3}{2})}{4\pi ^2} \sum_{n \geq 1} \frac{a(n)\overline{b(n)}}{n^{k + \frac{1}{2}}}.
  \end{equation*}
\end{theorem*}


\begin{remark}
  The main term in the above theorem, $CX^{\frac{1}{2}}$, mostly agrees with the main term for the sharp cut-off in~\cite{HafnerIvic89} as stated in ~\eqref{eq:OnAverageSquares} and~\eqref{eq:constantCfromHaffner} where $f=g$, though there are a few notable differences. Here we have normalized the Dirichlet series by dividing each $|S_f(n)|^2$ term by $n^{k-1}$, so the power of $X$ is $\frac{1}{2}$ rather than $k-1+\frac{1}{2}$. Furthermore, as we are performing a smooth cut-off, we have an extra factor of $\Gamma(\tfrac{3}{2})$ owing to the $\Gamma(s)$ in the inverse Mellin transform in~\eqref{eq:smoothcutofftransform}. It is not hard to check that performing a sharp cut-off on the non-normalized Dirichlet series by integrating against $1/s$ instead of $\Gamma(s)$  gives the exact same formula for the main term in ~\eqref{eq:constantCfromHaffner}.
\end{remark}

\begin{remark}
When proving the analogue of this theorem for $D(s, S_f \times \overline{S_g})$, there are only significant differences when $f$ and $g$ have nontrivial associated non-real nebentypus. As noted before case, the primary differences come from slightly more complicated Eisenstein series, $L$-functions associated to the Eisenstein series, and functional equations for these $L$-functions. We leave this out of the statement of the above theorem since for $\SL_2(\mathbb{Z})$ all automorphic forms are complex linear combinations of eigenforms with real coefficients, so the distinction between $S_f(n)\overline{S_g(n)}$ and $S_f(n)S_g(n)$ is trivial.
\end{remark}


\begin{remark}
  It is natural to try to shift the lines of integration further left, but this does not give much improvement. In Section~\ref{sec:Zswfg}, we see that $Z(s, 0, f\times g)$ has a line of poles when $\Re s = -\tfrac{1}{2} + \theta$, indicating that the exponent of the error term in Theorem~\ref{thm:second_moment} cannot be lowered.

  If one could prove a sharp cutoff instead of a smoothed sum of the above shape, one could prove the Classical Conjecture.
\end{remark}

\section{Cancellation within $W(s; f,g)$ for arbitrary level}\label{sec:genweight}

While the techniques and methodology of Section~\ref{sec:Zswfg} should work for general level, it is not immediately clear that the continuous spectrum of $Z(s,0,f\times g)$ will always perfectly cancel the leading pole and potential poles from zeta zeroes of $\frac{L(s,f\times g)}{\zeta(2s)}$.
In this section, we apply results of Chandrasekharan and Narasimhan (\cite{chandrasekharan1962functional} and~\cite{chandrasekharan1964mean}) that show that significant cancellation always occurs in the Mellin integrals, shedding new light on cancellation of terms between the diagonal and off-diagonal parts of shifted convolution sums and on the behavior of certain sums of Kloosterman zeta functions.

We then explicitly show that further cancellation holds for integral weight modular forms on $\Gamma_0(N)$ when $N$ is square-free, and completely generalize Theorem~\ref{thm:second_moment} to squarefree level.

\subsection{Leading Cancellation}

Suppose $f(z) = \sum a(n)e(nz)$ is a cusp form on the congruence subgroup $\Gamma$, of weight $k \in \mathbb{Z}\cup(\mathbb{Z} + \frac{1}{2})$ with $k > 2$. We also allow the possibility for $f$ to have nebentypus $\chi$ if $\Gamma=\Gamma_0(N)$ for some $N \in \mathbb{N}$. Theorem~1 of~\cite{chandrasekharan1964mean} gives that
\begin{equation}\label{eq:CN_compare}
  \frac{1}{X} \sum_{n \leq X} \frac{\lvert S_f(n) \rvert^2}{n^{k-1}}  = C X^{\frac{1}{2}} + O(\log^2 X).
\end{equation}

Let us compare the result of Chandrasekharan and Narasimhan to the methodology of this paper. Performing the decomposition of $D(s,S_f\times S_f)$ from Proposition~\ref{prop:decomposition} still leads us to study $Z(s, 0, f\times f)$ and $W(s; f,f)$.
The Rankin-Selberg convolution $L(s, f\times f)/\zeta(2s)$ has poles at $s = 1$ and at zeroes of $\zeta(2s)$ in $(0, \tfrac{1}{2})$.
The contributions from these poles must cancel with those from the poles of $Z(s, 0, f\times f)$ in the inverse Mellin transform from which Theorem~\ref{thm:second_moment} is derived, as otherwise the machinery of Sections~\ref{sec:Zswfg} and~\ref{sec:second_moment_of_sums} contradict~\eqref{eq:CN_compare}.
In particular, the leading contribution of the diagonal term cancels perfectly with a leading contribution from the off-diagonal,
\begin{equation*}
  \Res_{s = 1}\sum_{n \geq 1}\frac{\lvert a(n) \rvert^2}{n^{s + k - 1}} = - \Res_{s = 1}\sum_{n,h \geq 1} \frac{a(n) \overline{a(n-h)}+a(n-h)\overline{a(n)}}{n^{s + k - 1}}.
\end{equation*}
We investigate this cancellation further by sketching the arguments of Sections~\ref{sec:Zswfg} and~\ref{sec:second_moment_of_sums} in greater generality.

The spectral decomposition corresponding to Proposition~\ref{prop:spectralexpansionfull} is more complicated since we must now use the Selberg Poincar\'{e} series on $\Gamma$,
\begin{equation}
  P_h(z,s) := \sum_{\gamma \in \Gamma_\infty \backslash \Gamma_0(N)} \Im(\gamma z)^s e(h\gamma \cdot z).
\end{equation}
The spectral decomposition of $P_h$ will involve Eisenstein series associated to each cusp, $\mathfrak{a}$, of $\Gamma$, which each have an expansion,
\begin{equation}
  E_\mathfrak{a} (z,w) = \delta_\mathfrak{a} y^w + \varphi_\mathfrak{a}(0,w) y^{1-w} + \sum_{m \neq 0} \varphi_\mathfrak{a} (m,w) W_w(|m|z),
\end{equation}
where $\delta_\mathfrak{a} = 1$ if $\mathfrak{a} = \infty$ and is $0$ otherwise, where
\begin{align*}
  \varphi_{\mathfrak{a}}(0, w) &= \sqrt \pi \frac{\Gamma(w - \frac{1}{2})}{\Gamma(w)} \sum_c c^{-2w}S_\mathfrak{a}(0,0;c),
  \end{align*}
  and
   \begin{align*}
    \varphi_{\mathfrak{a}}(m, w) &= \frac{\pi^w}{\Gamma(w)} \lvert m \rvert^{w-1} \sum_c c^{-2w} S_\mathfrak{a}(0, n; c)
\end{align*}
when $m\neq 0$,
are generalized Whittaker-Fourier coefficients,
\begin{equation*}
  W_w(z) = 2\sqrt y K_{w - \frac{1}{2}}(2\pi y) e(x)
\end{equation*}
is a Whittaker function, $K_\nu(z)$ is a $K$-Bessel function, and
\begin{equation*}
S_\mathfrak{a}(m,n; c) = \sum_{\left(\begin{smallmatrix} a&\cdot \\ c&d \end{smallmatrix}\right) \in \Gamma_\infty \backslash \sigma_\alpha^{-1} \Gamma / \Gamma_\infty} e\left( m \frac{d}{c} + n \frac{a}{c}\right)
\end{equation*}
is a Kloosterman sum associated to double cosets of $\Gamma$ with
\[
\Gamma_\infty = \left \langle \begin{pmatrix} 1&n \\ &1 \end{pmatrix} : n \in \mathbb{Z} \right\rangle \subset \SL_2(\mathbb{Z}).
\]
This expansion is given in Theorem 3.4 of~\cite{iwaniec2002spectral}.

Letting $\mu_j$ be an orthonormal basis of the residual and cuspidal spaces, we may expand $P_h(z,s)$ by the Spectral Theorem (as presented in Theorem~15.5 of~\cite{IwaniecKowalski04}) to get
\begin{align}
  P_h(z,s) &= \sum_j \langle P_h(\cdot, s), \mu_j \rangle \mu_j(z) \\
           &\quad + \sum_\mathfrak{a} \frac{1}{4\pi} \int_\mathbb{R} \langle P_h(\cdot, s), E_\mathfrak{a}(\cdot, \tfrac{1}{2} + it)\rangle E_\mathfrak{a}(z, \tfrac{1}{2} + it) \ dt. \label{line:PoincareLevelNContinuous}
\end{align}
This is more complicated than the $\SL_2(\mathbb{Z})$ spectral expansion in~\eqref{eq:Pspectral} for two major reasons: we are summing over cusps and the Kloosterman sums within the Eisenstein series are trickier to handle. Continuing as before, we try to understand the shifted convolution sum
\begin{equation}
  Z(s,w,f\times f) = \frac{(4\pi)^{s + k - 1}}{\Gamma(s + k - 1)}\sum_{h \geq 1}\frac{\langle \mathcal{V}_{f,f}, P_h\rangle}{h^w}
\end{equation}
by substituting the spectral expansion for $P_h(z,s)$ and producing a meromorphic continuation.

The analysis of the discrete spectrum is almost exactly the same: it is analytic for $\Re s > -\frac{1}{2} + \theta$, and so the only new facet is understanding the continuous spectrum component corresponding to~\eqref{line:PoincareLevelNContinuous}. As noted above, the continuous spectrum of $Z(s, 0, f\times f)$ has leading poles that perfectly cancel the leading pole
of $L(s, f\times f) \zeta(2s)^{-1}$.
Using analogous methods to those in Section~\ref{sec:Zswfg}, we compute the continuous spectrum of $Z(s, 0, f\times f)$ to get
\begin{align}
  &\sum_{h \geq 1}\frac{(4\pi)^{s + k - 1}}{\Gamma(s + k - 1)}\sum_{\mathfrak{a}} \frac{1}{4\pi i} \int_{(\frac{1}{2})} \langle P_h(\cdot, s), E_\mathfrak{a}(\cdot, t)\rangle \langle \mathcal{V}_{f,f}, \overline{E_\mathfrak{a}(\cdot, t)}\rangle \, dt \nonumber \\
  \begin{split}
  & = \frac{(4\pi)^k}{\Gamma(s + k - 1)\Gamma(s)}\sum_{\mathfrak{a}} \int_{-\infty}^\infty  \left( \sum_{h,c \geq 1} \frac{S_\mathfrak{a} (0, h; c)}{h^{s + it} c^{1 - 2it}}\frac{\pi^{\frac{1}{2} - it}}{\Gamma(\frac{1}{2} - it)}\right) \times \\
  &\quad \times \Gamma(s - \tfrac{1}{2} + it)\Gamma(s - \tfrac{1}{2} - it) \langle \mathcal{V}_{f,f}, \overline{E_\mathfrak{a}(\cdot, \tfrac{1}{2} + it)} \rangle \, dt.
  \end{split}
\end{align}
We've placed parentheses around the arithmetic part, including the Kloosterman sums and factors for completing a zeta function that appears within the Kloosterman sums.

\begin{remark}
  We note also that~\eqref{eq:OnAverageSquares}, as in~\cite{chandrasekharan1962functional}, suggests that $D(s,S_f \times S_f)$ has a pole at $s = \frac{1}{2}$ and that no other poles should contribute additional terms to the inverse Mellin-transform for $\Re s > 0$.
  As in Section~\ref{ssec:meromorphic_continuation}, we recognize the pole at $s = \frac{1}{2}$ coming from the continuous spectrum of $Z(s, 0, f\times f)$.

  However, the exact nature of the potential polar behavior for $\Re s \in (0, \frac{1}{2})$ cannot be determined completely by comparison with~\cite{chandrasekharan1962functional}.
  It is natural to conjecture that there are no poles of $D(s, S_f \times S_f)$ with $\Re s > 0$ except for the known pole at $s = \frac{1}{2}$.
\end{remark}

We can make similar claims about the cancellation in the case of $D(s,S_f \times S_g)$ when $f \neq g$.
Indeed, if we let $h_1=f+g$ and $h_2=f+ig$, then we have that
\[
  |S_{h_1}(n)|^2 = |S_f(n)|^2+|S_g(n)|^2 +2\Re\left( S_f(n)\overline{S_g(n)}\right)
\]
and
\[
  |S_{h_2}(n)|^2 = |S_f(n)|^2+|S_g(n)|^2 +2\Im\left( S_f(n)\overline{S_g(n)}\right).
\]
Since $D(s,S_{h_i}\times S_{h_i}), \ D(s,S_f\times S_f), \ $ and $D(s,S_g\times S_g)$ do not have poles for $\Re(s) > \frac{1}{2}$, it follows that the continuations of
\[
  \sum_{n=1}^\infty \frac{\Re\left( S_f(n)\overline{S_g(n)}\right)}{n^{s+k-1}} \quad \text{and} \quad \sum_{n=1}^\infty \frac{\Im\left( S_f(n)\overline{S_g(n)}\right)}{n^{s+k-1}}
\]
are also meromorphic in this region.
Thus this is also the case for $D(s;S_f \times S_g)$.

We summarize the results of this section with the following theorem.

\begin{theorem}\label{thm:general_weight_level_comparison}
  Let $f(z) = \sum a(n)e(nz)$ and $g(z) = \sum b(n)e(nz)$ be weight $k > 2$ cusp forms on $\Gamma$, possibly with nebentypus $\chi$ if $\Gamma=\Gamma_0(N)$. Then
  \begin{align}\label{eq:general_level_general_comparison}
    \Res_{s = 1} \!\! \sum_{n \geq 1} \frac{a(n) \overline{b(n)}}{n^{s + k - 1}} = - \Res_{s = 1} \!\! \sum_{n,h \geq 1} \frac{a(n)\overline{b(n-h)}+a(n-h)\overline{b(n)}}{n^{s + k - 1}},
  \end{align}
  or equivalently
  \begin{align}\label{eq:general_level_kloosterman}
    &-\Res_{s = 1} \frac{L(s, f\times g)}{\zeta(2s)} =\\
    =& \Res_{s = 1}\sum_{\mathfrak{a}} \sum_{h=1}^\infty \frac{1}{4\pi} \int_\mathbb{R} \langle P_h(\cdot, s), E_\mathfrak{a}(\cdot, \tfrac{1}{2} + it)\rangle \langle \mathcal{V}_{f,g} , \overline{E_\mathfrak{a}(\cdot, \tfrac{1}{2} + it)}\rangle \, dt \\
    \begin{split}
    =& \Res_{s = 1} \frac{(4\pi)^k}{\Gamma(s + k - 1)\Gamma(s)}\sum_{\mathfrak{a}} \int_{-\infty}^\infty  \left( \sum_{h,c \geq 1} \frac{S_\mathfrak{a} (0, h; c)}{h^{s + it} c^{1 - 2it}}\frac{\pi^{\frac{1}{2} - it}}{\Gamma(\frac{1}{2} - it)}\right) \\
     &\quad \times \Gamma(s - \tfrac{1}{2} + it)\Gamma(s - \tfrac{1}{2} - it) \langle \mathcal{V}_{f,g}, \overline{E_\mathfrak{a}(\cdot, \tfrac{1}{2} + it)} \rangle \, dt.
    \end{split} \label{split}
  \end{align}
\end{theorem}

\begin{remark}
  The result is essentially the same and the above argument is largely unchanged if we replace $\overline{b(n)}$ by$b(n)$.
  The only potential issue is if there is an accompanying non-real nebentypus, $\chi$, for $f$ and $g$.
  In this case, we change the group we are considering to $\Gamma_1(N)$, where the nebentypus does not affect the proof.
  Since $g$ and $\overline{T_{-1}g}$ are both weight $k$ holomorphic cusp forms for $\Gamma_1(N)$, we use that $D(s;S_f \times \overline{S_g})=D(s;S_f \times S_{\overline{T_{-1}g}})$ and we get that $D(s;S_f \times \overline{S_g})$ is also meromorphic for $\Re(s) > \frac{1}{2}$.
\end{remark}

\subsection{Complete Cancellation for Square-free Level}

Now we consider the case where the level $N$ is square-free, where our methodology explicitly demonstrates the same complete, remarkable cancellation as in level $1$.
Let $\Gamma=\Gamma_0(N)$ where $N$ is square-free.
A complete set of cusp representatives is given by $\{ \frac{1}{v} : \, v \mid N\}$.
As in~\cite{deshouillers1982kloosterman}, we can derive more explicit formulas for $\varphi_{\mathfrak{a}}(m,w)$ when $\mathfrak{a}=\frac{1}{v}$,
\begin{align*}
  \varphi_{\mathfrak{a}}(0,w)=\frac{\zeta^*(2s-1)}{\zeta^*(2s)} \left( \frac{1}{vN} \right)^s \varphi(v) \prod_{p | N} \left(1-\frac{1}{p^{2s}}\right)^{-1} \prod_{p | \frac{N}{v}} \left(1-\frac{1}{p^{2s-1}}\right)
\end{align*}
and
\begin{align*}
\varphi_{\mathfrak{a}}(m,w)= \frac{2\pi^s |m|^{s-\frac{1}{2}}}{\Gamma(s)} \left( \frac{1}{vN} \right)^s \frac{\sigma_{1-2s}^{(N)}(m)}{\zeta^{(N)}(2s)} \prod_{\substack{p | v \\ p^\alpha \| m}} (p \sigma_{1-2s}\left(p^{\alpha - 1}) - \sigma_{1-2s}(p^\alpha)\right),
  \end{align*}
  where
  \[
  \sigma_v^{(c)}(n) = \sum_{\substack{d | n \\ (d,c) = 1}} d^v
  \]
  and
  \[
  \zeta^{(N)}(2s)= \zeta(2s) \prod_{p | N} (1-p^{-2s}).
  \]
  From Hejhal~\cite{Hejhal}, we know that the Eisenstein series satisfy the functional equation
  \begin{equation}\label{eq:general_level_eisenstein_functional}
    \begin{split}
      &E_\infty(z,s) =\\
      &\quad\frac{\zeta^*(2 - 2s)}{\zeta^*(2s)}\prod_{p \mid N} \frac{1}{1 - p^{2s}} \sum_{v \mid N} \prod_{p \mid v} (1 - p)\prod_{p \mid \frac{N}{v}} (p^{1-s} - p^s)E_{\frac{1}{v}}(z, 1-s).
    \end{split}
  \end{equation}

Using these coefficients, we can explicitly describe write the continuous part of the spectrum of $Z(s, w, f\times g)$ in~\eqref{split} as
\begin{align}\label{eq:general_level_specific_continuous}
    &\frac{(4\pi)^k}{\Gamma(s+k-1)} \frac{1}{4\pi i} \int_{(0)} \sum_{v | N} \langle V_{f,g}, E_{\frac{1}{v}}(*,\frac{1}{2}-\overline{z}) \rangle \frac{\Gamma(s-\frac{1}{2} + z)\Gamma(s-\frac{1}{2}-z)}{\Gamma(s)(vN)^{\frac{1}{2}+z}} \notag \\
    & \times  \frac{\zeta(s+w-\frac{1}{2}-z) \zeta^{(\frac{N}{v})}(s+w-\frac{1}{2}+z)}{\pi^{-(\frac{1}{2}+z)} \Gamma(\frac{1}{2}+z) \zeta^{(N)}(1+2z)} \prod_{p | v} (p^{\frac{3}{2} + z -s -w} - 1) \,dz.
\end{align}

For $\Re (s+w)$ close to $\frac{3}{2}$, we obtain residual terms from the poles at $z = \pm(s+w-\frac{3}{2})$ through the same procedure for moving lines of integration from Section~\ref{sec:Zswfg}.
For $z=s+w-\frac{3}{2}$, there is a pole only when the product over primes dividing $v$, as otherwise the simple pole from the zeta function is cancelled by the zero of the product.
Therefore, a pole occurs only when $v=1$.
In that case, the $z$ residue is:
\begin{equation} \label{polepart1}
  \begin{split}
    \frac{(4\pi)^k}{4\Gamma(s+k-1)} &\langle V_{f,g}, E_1(*,2-\overline{s}-\overline{w}) \rangle \times \\
    &\frac{\Gamma(2s+w-2) \Gamma(1-w)}{\Gamma(s+w-1)\Gamma(s)} \frac{\pi^{s+w-1}}{N^{s+w-1}}.
  \end{split}
 \end{equation}
Setting $w=0$ and then taking residue at $s=1$, we get
\begin{equation*}
  -\frac{(4\pi)^k}{4\Gamma(k)} \Res_{s=1} \langle V_{f,g}, E_1(*,\overline{s}) \rangle = -\frac{1}{2} \Res_{s=1} \sum_{n \geq 1} \frac{a(n)\overline{b(n)}}{n^{s+k-1}}.
\end{equation*}

When $z=\frac{3}{2}-s-w$, there is a pole from the zeta function in the numerator with residue
\begin{equation}
  \label{polepart2}
  \begin{split}  \frac{(4\pi)^k}{2\Gamma(s+k-1)} &\sum_{v | N} \langle V_{f,g}, E_{\frac{1}{v}}(*,\overline{s}+\overline{w}-1) \rangle \frac{\Gamma(1-w)\Gamma(2s+w-2)}{\Gamma(s) (vN)^{2-s-w}} \\
  &\times \frac{\zeta(2s+2w-2) \prod_{p | \frac{N}{v}} (1-\frac{1}{p}) \prod_{p | v} (p^{3-2s-2w} - 1)}{\pi^{-(2-s-w)} \Gamma(2-s-w) \zeta^{(N)}(4-2s-2w)}.
  \end{split}
\end{equation}
We set $w=0$.
Using the functional equation from Hejhal~\cite{Hejhal}, which for square-free $N$ gives
\begin{equation*}
  E_\infty(z,s) = \frac{\zeta^*(2-2s)}{\zeta^*(2s)} \prod_{p | N} \frac{1}{1-p^{2s}} \sum_{v | N} \prod_{p | v}(1-p) \prod_{p | \frac{N}{v}}(p^{1-s}-p^s) E_{\frac{1}{v}}(z,1-s)
 \end{equation*}
and
\begin{equation*}
  E_{1}(z,s) = \frac{\zeta^*(2-2s)}{\zeta^*(2s)} \prod_{p | N} \frac{1}{1-p^{2s}} \sum_{v | N} \prod_{p | \frac{N}{v}}(1-p) \prod_{p | v}(p^{1-s}-p^s) E_{\frac{1}{v}}(z,1-s).
\end{equation*}
Notice that this is the same as the residue at $z=s+w-\frac{3}{2}$, so we can explicitly verify the cancellation of poles of $Z(s, 0, f\times g)$ at $s = 1$ as guaranteed by Theorem~\ref{thm:general_weight_level_comparison}.

From the continuous spectrum~\eqref{eq:general_level_specific_continuous}, there are residual terms coming from $z = \pm (s + w - \frac{1}{2})$.
In Section~\ref{sec:Zswfg}, setting $w = 0$ and examining the corresponding residual term in the level $1$ case led to a term perfectly cancelling the diagonal contribution $L(s, f\times f)/\zeta(2s)$.

We now investigate this residual term and cancellation in the square-free level case.
Recall that we now have $w = 0$.
The residual term from $z = \frac{1}{2} - s$, occurring when $z = \frac{1}{2} - s$ in the zeta function $\zeta^{(\frac{N}{v})}(s - \frac{1}{2} + z)$ in the numerator of~\eqref{eq:general_level_specific_continuous}, vanishes except when $v = N$ since $\zeta^{(\frac{N}{v})}(0) = 0$ if $\frac{N}{v} > 1$.
The residual term when $v = N$ is
\begin{equation*}
  -\frac{(4\pi)^k}{4\Gamma(s+k-1)} \langle V_{f,g}, E_\infty(*,\overline{s}) \rangle \frac{\pi^{1-s}N^{2s-2}\Gamma(2s-1)\zeta(2s-1) }{\Gamma(s) \Gamma(1-s) \zeta^{(N)}(2-2s)}  \prod_{p | N} (p^{2-2s} - 1).
\end{equation*}
Using duplication formula and functional equation for the completed zeta function, this simplifies to
\begin{equation}\label{eq:final_continuous_cancellation_part1}
 -\frac{1}{2} \sum_{n \geq 1} \frac{a(n) \overline{b(n)}}{n^{s+k-1}}.
 \end{equation}
 (Note that this is very similar to the analysis of~\eqref{eq:secondresidual_example}).

 The residual term from $z = s-\frac{1}{2}$ coming from the other zeta function in the numerator of~\eqref{eq:general_level_specific_continuous} is
\begin{equation*}
 \frac{(4\pi)^k}{2\Gamma(s+k-1)} \sum_{v | N} \langle V_{f,g}, E_{\frac{1}{v}}(*,1-\overline{s}) \rangle \frac{\Gamma(2s-1)}{\Gamma(s) (vN)^s} \frac{\zeta(0) \zeta^{(\frac{N}{v})}(2s-1)}{\pi^{-s} \Gamma(s) \zeta^{(N)}(2s)} \prod_{p | v} (p-1).
 \end{equation*}
 Using the functional equation~\eqref{eq:general_level_eisenstein_functional} of the Eisenstein series and the relation
\begin{align*}
  &\frac{\zeta(0) \zeta^{(\frac{N}{v})}(2s-1)}{\pi^{-s}(vN)^s  \zeta^{(N)}(2s)} \prod_{p | v} (p-1) \\
  &\qquad= \frac{\zeta(2s-1)}{(vN)^s\zeta^*(2s)} \frac{\prod_{p | v} (p-1) \prod_{p | \frac{N}{v}} (1-p^{1-2s})}{\prod_{p | N} (1-p^{-2s})} \\
  &\qquad= \frac{\pi^{s-\frac{1}{2}} }{\Gamma(s-\frac{1}{2})} \frac{\zeta^*(2-2s)}{\zeta^*(2s)} \frac{\prod_{p | v} (p-1) \prod_{p | \frac{N}{v}} (p^s-p^{1-s})}{\prod_{p | N} (p^{2s} - 1)},
\end{align*}
we have
\begin{align*}
  \frac{(4\pi)^k}{2\Gamma(s+k-1)} &\sum_{v | N} \langle V_{f,g}, E_{\frac{1}{v}}(*,1-\overline{s}) \rangle \frac{\Gamma(2s-1)}{\Gamma(s) (vN)^s} \frac{\zeta(0) \zeta^{(\frac{N}{v})}(2s-1)}{\pi^{-s} \Gamma(s) \zeta^{(N)}(2s)} \prod_{p | v} (p-1)\notag \\
  &= -\frac{(4\pi)^k}{4\Gamma(s+k-1)} \frac{\pi^{s-1}}{4^{1-s}} \langle V_{f,g}, E_{\infty}(*,\overline{s}) \rangle = -\frac{1}{2} \sum_{n \geq 1} \frac{a(n) \overline{b(n)}}{n^{s+k-1}}.
\end{align*}
This is exactly the same as~\eqref{eq:final_continuous_cancellation_part1}, and together these add together to perfectly cancel $L(s, f\times g) \zeta(2s)^{-1}$.
Thus we see these residual terms in $Z(s,0,f \times g)$ perfectly cancel the Rankin-Selberg diagonal for $\Re s \leq \frac{1}{2}$.

Finally, as in Section~\ref{sec:Zswfg}, we conclude that $D(s, S_f \times S_g)$ is analytic for $\Re(s) > -\frac{1}{2} + \theta$, except for a simple pole at $s = \frac{1}{2}$ coming from~\eqref{polepart1} and~\eqref{polepart2}.
Here, $\theta = \max_j \Im(t_j) \leq \frac{7}{64}$ denotes the progress towards Selberg's Eigenvalue Conjecture.

In other words, the meromorphic properties of $D(s, S_f \times S_g)$ for square-free level are very similar to the meromorphic properties on level $1$.
It is now clear that performing the exact same Mellin transform analysis as in Section~\ref{sec:second_moment_of_sums} gives the following generalization of Theorem~\ref{thm:second_moment}.

\begin{theorem}\label{thm:second_moment2}
  Suppose that $f$ and $g$ are holomorphic cusp forms of integral weight $k$ on $\Gamma_0(N)$, $N$ square-free.
  Then for any $\epsilon > 0$,
  \begin{equation} \label{maineq1}
    \frac{1}{X} \sum_{n \geq 1}\frac{S_f(n)\overline{S_g(n)}}{n^{k - 1}}e^{-n/X} = CX^{\frac{1}{2}} + O_{f,g,\epsilon}(X^{-\frac{1}{2} + \theta + \epsilon})
  \end{equation}
  where $C$ is is a calculable residue depending only on $f$ and $g$, and $\theta = \max_j \Im (t_j)$ denotes the progress towards Selberg's Eigenvalue Conjecture for $\Gamma_0(N)$.
\end{theorem}

\section*{Acknowledgements}

We would like to thank Jeff Hoffstein at Brown University, who introduced us to this problem by asking a question at a talk given by Winfried Kohnen at Jeff's 61st birthday conference, and who has besides mentored us all.

\vspace{20 mm}
\bibliographystyle{alpha}
\bibliography{compiled_bibliography}

\end{document}